\documentclass[a4paper,11pt]{article}
\usepackage{amsmath,amssymb,amsthm,stmaryrd} 
\usepackage{mathtools}
\usepackage{color}
\usepackage{enumerate}
\usepackage[dvipsnames]{xcolor}
\usepackage[hyperindex]{hyperref}

\usepackage[normalem]{ulem}   

\numberwithin{equation}{section}

\theoremstyle{theorem}
\newtheorem{theorem}{Theorem}[section]
\newtheorem{lemma}[theorem]{Lemma}

\newtheorem{proposition}[theorem]{Proposition}
\newtheorem{corollary}[theorem]{Corollary}
\theoremstyle{definition}
\newtheorem{definition}[theorem]{Definition}
\newtheorem{remark}[theorem]{Remark}

\newtheorem{observations}[theorem]{Observations}

\renewcommand{\a}{\alpha}

\renewcommand{\d}{\delta}

\newcommand{\x}{\xi}

\newcommand{\sub}{\subset}

\newcommand{\R}{\mathbb{R}}
\newcommand{\N}{\mathbb{N}}

\newcommand{\Z}{\mathbb{Z}}

\newcommand{\MM}{\mathcal{M}}
\newcommand{\NN}{\mathcal{N}}
\newcommand{\FF}{\mathcal{F}}

\newcommand{\PP}{\mathcal{P}}

\renewcommand{\SS}{\mathcal{S}}
\newcommand{\BB}{\mathcal{B}}
\newcommand{\GG}{\mathcal{G}}
\newcommand{\we}{\wedge}

\newcommand{\DD}{\mathcal{D}}

\newcommand{\hel} {
\hskip2.5pt{\vrule height7pt width.5pt depth0pt}
\hskip-.2pt\vbox{\hrule height.5pt width7pt depth0pt}
\, }
\newcommand{\restr}{\hel}

\renewcommand{\j}{\jmath}
\renewcommand{\i}{\imath}

\newcommand{\mc}{\mathcal}

\newcommand{\sm}{\backslash}
\newcommand{\pt}{\partial}

\newcommand{\eps}{\varepsilon}

\newcommand{\oo}{\infty}

\newcommand{\nb}{\nabla}
\newcommand{\ov}{\overline}

\newcommand{\wt}{\widetilde}
\newcommand{\h}{\mathcal{H}}
\newcommand{\dw}{\downarrow}
\newcommand{\up}{\uparrow}
\newcommand{\cd}{\cdot}
\newcommand{\longto}{\longrightarrow}

\renewcommand{\t}{\times}

\newcommand\un{\bs{1}}
\newcommand{\Om}{\Omega}
\newcommand{\om}{\omega}

\newcommand{\lb}{\llbracket}
\newcommand{\rb}{\rrbracket}
\newcommand{\be}{\begin{equation}}
\newcommand{\ee}{\end{equation}}

\newcommand{\sslo}{\underset{\text{\scriptsize lex.}}<}

\newcommand{\XXint}[3]{{\setbox0=\hbox{$#1{#2#3}{\int}$}
      \vcenter{\hbox{$#2#3$}}\kern-.5\we0}}

\newcommand{\void}{\varnothing}

\newcommand{\st}{\stackrel}
\newcommand{\lt}{\left}
\newcommand{\rt}{\right}
\newcommand{\bs}{\boldsymbol}
\newcommand{\M}{\mathbb{M}}
\newcommand{\F}{\mathbb{F}}

\DeclareMathOperator{\supp}{supp}

\DeclareMathOperator{\bndry}{fr}
\DeclareMathOperator{\Gr}{Gr}

\newcommand{\blue}{\color{blue}}

\newcommand{\comfig}[1]{}

\title{ Set-decomposition of normal rectifiable $G$-chains \emph{via} an abstract decomposition principle}
\author{M. Goldman\footnote{ CMAP, CNRS, \'Ecole polytechnique, Institut Polytechnique de Paris, 91120 Palaiseau,
France, email: michael.goldman@cnrs.fr} \and B. Merlet\footnote{Univ. Lille, CNRS, UMR 8524, Inria - Laboratoire Paul Painlev\'e, F-59000 Lille, email: benoit.merlet@univ-lille.fr}}

\begin{document}
\maketitle

\begin{abstract}
We introduce the notion of \emph{set-decomposition} of a normal $G$-flat chain $A$ in $\R^n$ as a sequence $A_j=A\restr S_j$ associated to a Borel partition $S_j$ of $\R^n$ such that $\N(A)=\sum \N(A_j)$.
We show  that any \emph{normal rectifiable} $G$-flat chain admits a decomposition in set-indecomposable sub-chains. This generalizes the decomposition of sets of finite perimeter in their ``measure theoretic" connected components due to Ambrosio, Caselles, Masnou and Morel. It can also be seen as a variant of the decomposition of integral currents in indecomposable components by Federer.\\
As opposed to previous results, we do not assume  that $G$ is boundedly compact. Therefore we cannot rely on the  compactness of sequences of chains with uniformly bounded $\N$-norms. We deduce instead  the result from a new abstract decomposition principle. 

As in earlier proofs a central ingredient is  the validity of an isoperimetric inequality. We obtain it here using the finiteness of some $h-$mass to replace integrality. 
\end{abstract}


\section{Introduction}
The aim of this note is to extend the notion of decomposition of normal currents from the integral setting~\cite{Federer, ACMM, BPR2020,BdNP2022} to the general setting of normal rectifiable $G-$flat chains. This work is motivated by~\cite{GM_tfc} where we use  the decomposition result to study the rectifiability properties of tensor flat chains.\\

In order to state our main result, let us start with some notation and definitions. Let $G$ be a complete Abelian normed group and let $0\le k\le n$. We denote by $\FF_k^G(\R^n)$ the group of $k$-chains in $\R^n$ with coefficients in $G$ as introduced by Fleming in~\cite{Fleming66}. However, as in~\cite{White1999-1,White1999-2} we do not assume that chains are compactly supported. The mass of a chain $A$ is denoted $\M(A)$ and $\MM_k^G(\R^n)$ is the subgroup of finite mass $k$-chains. The restriction of $A\in\MM_k^G(\R^n)$ to a Borel set $S\sub\R^n$ is denoted $A\restr S$. By definition, $A\in\MM_k^G(\R^n)$ is rectifiable if $A=A\restr \Sigma$ for some countably $k$-rectifiable set $\Sigma\sub\R^n$. By~\cite[Section 6]{White1999-1}, we can identify every rectifiable $k-$chain with a measure $w \xi \h^k\restr\Sigma$
where $w: \R^n\to G$ is  Borel measurable and $\xi$ is a Borel measurable field of unit simple $k$-vectors orienting $\Sigma$. Eventually, we set $\N(A):=\M(A)+\M(\pt A)$ and denote $\NN_k^G(\R^n)=\{A\in \FF_k^G(\R^n):\N(A)<\oo\}$  the subgroup of normal $k$-chains.

Flat chains with real or integer coefficients were introduced as a particular class of currents in~\cite{FedFlem60}. Later, in~\cite{Fleming66} Fleming proposed a theory for flat chains with coefficients in a (commutative normed) group, still in an Euclidean ambient space, which have been further developed by White in~\cite{White1999-1,White1999-2}. Shortly after,~\cite{AK2000} introduced a theory of currents in an ambient metric space. Both this theory and the theory of Fleming and White have been generalized to form a theory of $G$-flat chains in metric spaces in~\cite{HdP2012}.\\
Even in the case of Euclidean ambient spaces, the topological and geometrical structure of rectifiable and even finite mass chains is still under investigation (see \textit{e.g.}~\cite{AlMa2017,Young2018,flat}). In this note we introduce the notion of set-decomposition of normal flat chains and prove that every normal and  rectifiable chain can be decomposed in indecomposable components.

\begin{definition}
Let $A\in\NN_k^G(\R^n)$.\medskip

\noindent 
(1) A \emph{set-decomposition} of $A$ is a sequence (finite or countable) of normal chains $A_j$ such that there exists a  Borel partition $S_j$ of $\R^n$ with  $A_j=A\restr S_j$ for every $j$ and $\N(A)=\sum \N(A_j)$.\\ 
We say that each $A_j$ is a \emph{set-subchain} of $A$.\medskip

\noindent
(2) We say that $A$ is \emph{set-indecomposable} if the only set-decompositions of $A$ are trivial, that is, for any set-decomposition $A_j$ of $A$, there holds $A_j=A$ for some index $j$ and $A_j=0$ for the others.
\end{definition}
\begin{remark}\label{rem:rectif}
 Notice that by definition if $A$ is rectifiable then for every Borel set $S$, $A\restr S$ is also rectifiable. In particular every set-decomposition of a rectifiable chain is made of rectifiable subchains.
\end{remark}

\begin{theorem}\label{thm_decomp}
Let $A\in\NN_k^G(\R^n)$, if $A$ is rectifiable then it admits a set-decomposition in set-indecomposable subchains. 
\end{theorem}
Such decomposition is also called a maximal set-decomposition of $A$. 

We obtain Theorem~\ref{thm_decomp} as a corollary of the abstract decomposition Lemma~\ref{lem_decomp}, stated and established in Section~\ref{S_decomp}. Let us give some comments. \medskip

\noindent
(a) When $A$ is an integral current, a set-decomposition is in general coarser than a decomposition into indecomposable integral currents introduced by Federer~\cite[4.2.25]{Federer}.  For instance, if $A'$ is the integral 1-current with multiplicity 1 associated with a smooth oriented Jordan curve, then $A:=2A'$ is set-indecomposable in $\NN^\Z_1(\R^n)$ but admits the decomposition $(A',A',0,\dots)$ in the sense of Federer. \medskip

\noindent
(b) In the case  $k=n$ and $G=\Z$, if $A=\lb E\rb$, where $E$ is a set of finite perimeter, our definition corresponds to the decomposition of $E$ into its measure theoretic connected components introduced in~\cite{ACMM} and Theorem~\ref{thm_decomp} generalizes~\cite[Theorem~1]{ACMM}. \medskip

\noindent
(c)  For $k=0$, the set-decomposition in set-indecomposable subchains of a normal rectifiable $0$-chain is essentially unique.  Indeed, any normal rectifiable $0$-chain is of the form $A=\sum g_j\lb x^j\rb$ where $g_j\in G$ is such that $\sum |g_j|_G<\oo$ and $x^j\in\R^n$ is a sequence of \emph{pairwise distinct} points. The set-indecomposable $0$-chains are the chains $g\lb x\rb$ for $g\in G$, $x\in\R^n$. It follows that the sequence $(g_1\lb x^1\rb,g_2\lb x^2\rb,\dots)$ is a set-decomposition in set-indecomposable subchains of the above $0$-chain $A$. Moreover, all the maximal set-decompositions are obtained by rearranging this latter and possibly inserting and removing zeros.\medskip

At the other end, any $n$-chain $A$ is rectifiable and the group of normal rectifiable $n$-chains is the group of normal $n$-chains. We believe that the set-decomposition of normal $n$-chains in set-indecomposable subchains is also essentially unique. We establish this fact in the particular case $G=\R$ (and thus also $G=\Z$), see the statement of Proposition~\ref{prop_uniq_n}.
\medskip

On the contrary, for $1\le k\le n-1$, the decomposition in set-indecomposable subchains is in general not unique even up to rearrangements. For instance, set $n=2$ and $G=\Z$ and consider the polyhedral $1$-chains with multiplicity 1, $A^h$ and $A^v$, where
\begin{enumerate}[$*$]
\item $A^h$ is supported by the horizontal segment $[-1,1]\t\{0\}$ and is oriented by $e_1$, 
\item $A^v$ is supported by the vertical segment $\{0\}\t[-1,1]$ and is oriented by $e_2$.
\end{enumerate}
Setting, 
\[A:=A^h+A^v,\qquad\qquad  A^+:=A\restr\{x_2>x_1\},\qquad\qquad A^-:=A\restr\{x_2<x_1\},\] 
we see that $(A^h,A^v,0,\dots)$ and $(A^+,A^-,0,\dots)$ are  two distinct set-decompositions of $A$ in set-indecomposable subchains.\medskip

\noindent
(d)  For a decomposition process to result in an \emph{at most countable number} of indecomposable parts, we need some principle which prevents big pieces  from crumbling into dust. In the abstract decomposition Lemma~\ref{lem_decomp}, this principle is provided by assumption (H2) which is a \emph{superlinear} estimate of a ``weak norm'' by a ``strong norm''.  In~\cite{Federer,ACMM} this role is played by some isoperimetric inequalities which give a \emph{superlinear} estimate of the mass of an object by the mass  of its boundary. Here we use Lemma~\ref{lem_isopineq} which is of the same nature. Indeed, with Step~1 of the proof of Theorem~\ref{thm_decomp}, we have that if $A$ is normal and rectifiable there exists an increasing and strictly subadditive cost function $h\in C(\R_+,\R_+)$ with $h'(0^+)=\oo$ such that $\M_h(A)<\oo$ (see Definition~\ref{def_hmass} for the definition of $\M_h$). The isoperimetric inequality of Lemma~\ref{lem_isopineq} (which extends Almgren's isoperimetric inequality~\cite{Almgren}, see Remark~\ref{rem_isoper}) then provides a nondecreasing function $\eta:\R_+\to\R_+$ such that $\eta(m)\to0$ as $m\dw0$ and for every $k$-chain $A'$
\be\label{eta_intro}
\F(A')\le \eta(\M(A'))\,(\M_h(A')+\N(A')).\smallskip
\ee

\noindent
(e) Unlike the references mentioned above, our proof of Theorem~\ref{thm_decomp} does not use any compactness theorem of the form:
\be\label{notassumed}
\text{``for $\Lambda\ge0$, the set $\{A\in \NN_k^G(\R^n): \supp A\sub \ov B_\Lambda,\ \N(A)\le \Lambda\}$ is compact in $\F$-norm.''}
\ee
This statement is true if and only if $G$ is boundedly compact in which case it is a classical consequence of the deformation theorem. We use here instead the convergence in strong norm of monotone sequences. More precisely we use the following simple fact: if $A$ has finite mass (resp. finite $h$-mass) and $A_j=A\restr S_j$ with $S_j$ a nonincreasing sequence of Borel subsets of $\R^n$ we have by the monotone convergence theorem, $A_j\to A\restr \cap S_j$ in mass (resp. in $h$-mass).\medskip


\noindent
(f) Some generalizations of the results of~\cite{Federer,ACMM} exist in the context of real currents in metric spaces of~\cite{AK2000}. Namely,~\cite[Theorem~2.14]{BPR2020} generalizes the decomposition of sets with finite perimeter in a (doubling) metric measure space and~\cite[Theorem~3.2]{BdNP2022} generalizes the decomposition of integral currents in metric spaces. This suggests that a version of Theorem~\ref{thm_decomp} should hold true within the theory of $G$-currents in metric spaces developed in~\cite{HdP2012}.\medskip

\noindent 
(g) In connection with (d), let us point out that if we fix a cost function $h$ as above (satisfying in particular $h'(0^+)=\infty$) we can define the $h-$mass of any flat chain as the lower-semicontinuous envelope  of $\M_h$ restricted to polyhedral chains. By~\cite[Theorem 8.1]{White1999-2} every normal chain with finite $h-$mass is rectifiable (with $h-$mass coinciding with $\M_h$ by~\cite{White1999-1,CdRMS2017}). Therefore Theorem~\ref{thm_decomp} provides a decomposition in indecomposable components for normal chains of finite $h-$mass. Partly due to their connection with branched transport models, this type of functionals has received a lot of attention in the past few years, see \textit{e.g.}~\cite{BraWir,CFM2019a,colombo2021well}. It is however worth noticing that our notion of set-decomposition (and indecomposability) is independent of the choice of $h$.\medskip

\noindent
(h) Let us mention similar decomposition results for rectifiable $m$-varifolds in $U\subset\R^d$ whose first variation is a measure. First in~\cite[\S 6.12]{Menne2016} the existence of a set-decomposition in set-indecomposable components is established. Second, in~\cite{Chou22} the decomposition of an integral varifold into countably many indecomposable integral varifolds is proved.\\ 
In the first case a varifold $V$ is said to be set-decomposable if there exists a Borel subset $B\sub U$ such that $W:=V\restr B\t[\Gr(m,\R^n)]$ satisfies $W\not\in\{0,V\}$  and $\d W=(\d V)\restr B$. This is stronger (and in general strictly stronger) than the condition $\|\d V\|=\|\d W\|+\|\d(V-W)\|$. In the second case $V$ is decomposable if there is exists an integral varifold $W\le V$ such that $W\not \in\{0,V\}$ and  $\|\d V\|=\|\d W\|+\|\d(V-W)\|$. \\
We believe that these results (including the variant with a weaker notion of set-decomposability) could be obtained as applications of Lemma~\ref{lem_decomp}. In this setting the superlinear estimate akin to the isoperimetric inequality should stem from the monotonicity identity of~\cite[\S 4.5\,\&\,4.6]{Menne2016}. However, using Lemma~\ref{lem_decomp} is not likely to improve the results of~\cite{Menne2016,Chou22} or even simplify their proofs. Consequently, we opt not to investigate further these issues here.\bigskip

The main contribution of this note is the fact that we obtain the decomposition result Theorem~\ref{thm_decomp} without the closure/compactness property~\eqref{notassumed}, that is: without assuming that $G$ is boundedly compact.\\
To highlight the interest of our method and how it differs from previous approaches, we give in Appendix~\ref{App2} an alternative proof of the theorem under the additional assumption that (1.2) holds true. This alternative proof is very close in spirit to the one of~\cite[Theorem~1]{ACMM}.\bigskip

In the next section we establish Lemma~\ref{lem_decomp}. It provides an abstract decomposition principle in Abelian normed groups, assuming a general version of~\eqref{eta_intro} and a closure property for nonincreasing sequences in the subset chosen for the decompositions. In Section~\ref{S_proof_td}, we prove the isoperimetric inequality for normal rectifiable chains (Lemma~\ref{lem_isopineq}) and then Theorem~\ref{thm_decomp}. 
In Section~\ref{S_uniq_n} we state and prove Proposition~\ref{prop_uniq_n} about the uniqueness of the maximal set-decomposition of a normal $n$-chain when $G$ is a subgroup of $(\R,+)$. \\
In Appendix~\ref{App1}, we establish a simple ``higher integrability" lemma used in the proof of Theorem~\ref{thm_decomp}. Eventually, in Appendix~\ref{App2}, we give a more classical proof of Theorem~\ref{thm_decomp} valid when $G$ is boundedly compact. 

\section{An abstract decomposition Lemma}
\label{S_decomp}

\begin{lemma}\label{lem_decomp}
Let $(\GG,+,\nu)$ be a complete Abelian normed group and let $\SS\sub\GG$ such that $0\in\SS$.
\begin{enumerate}[(i)]
\item  A sequence $a_j\in\SS$ is a decomposition of $b\in\GG$ if $b=\sum a_j$ and $\nu(b)=\sum\nu(a_j)$. In such case we write $a_j\preceq b$, for every $j\ge1$.
\item $b$ is an atom if $b\in\SS$ and any decomposition of $b$ is trivial, that is $a\preceq b$ implies $a=0$ or $a=b$.
\end{enumerate}
We make the following assumptions.
\begin{enumerate}
\item[(H1)] The limit of nonincreasing sequences $b_j\in\GG$ (that is $b_1\succeq b_2\succeq\dots$) belong to $\SS$.
\item[(H2)] There exists another norm $\phi$ on $\GG$ and a nondecreasing function $\eta:\R_+\to\R_+$ with $\lim\limits_{s\to0}\eta(s)=0$ such that $\phi(a)\le \eta(\nu(a))\,\nu(a)$ for every $a\in\GG$. 
\end{enumerate}
Then, if $b\in\GG$ admits at least one decomposition, it admits a decomposition in atoms.
\end{lemma}

\begin{remark}\label{rem_lem_decomp}~\medskip

\noindent (1) Since $0\in\SS$ any element $a\in\SS$ admits the trivial decomposition $(a,0,0,\dots)$ and the lemma implies that under Assumptions (H1)\&(H2) any element of $\SS$ admits a decomposition in atoms.\footnote{Conversely if any element of $\SS$ admits a decomposition in atoms, then if $b\in\GG$ admits a decomposition $(b_1,b_2,\dots)$ we obtain a decomposition of $b$ by collecting the decompositions in atoms of the $b_j$'s.}\\
Also notice that $0$ is always an atom.\medskip

\noindent (2) Let us stress again that~(H1) is the only closure property that we consider in the lemma, and that it only concerns monotone sequences. In the proof of Theorem~\ref{thm_decomp}, (H1) follows from the monotone convergence theorem of measure theory.\medskip

\noindent (3) As already mentioned, the lemma provides an alternative proof to the existence of decompositions in indecomposable components of a normal integral current supported in some compact $K$ (\cite[4.2.25]{Federer},~\cite[Theorem~3.2]{BdNP2022}).  For this, we take  $\SS=\GG$ as the group of normal integral currents supported in $K$, $\nu=\N$ and $\phi=\F$. In this setting,~(H1) follows from the completeness of $(\GG,\nu)$  and~(H2) from the isoperimetric inequality.
\end{remark}

Before proving the lemma, let us discuss some consequences of the assumptions. It is convenient to consider a broader notion of decomposition. For this, we complete the definitions (i)(ii) in the lemma by:
\begin{enumerate}[($\a$)]
\item[(iii)] We say that a sequence $a_j\in\GG$ is a pseudo-decomposition of $b\in\GG$ if $b=\sum a_j$ and $\nu(b)=\sum\nu(a_j)$. In such case we write for every $j\ge1$, $a_j\preceq_\psi b$.\medskip
\end{enumerate}
A decomposition is then a pseudo-decomposition whose components lie in $\SS$.
\begin{observations}\label{obs_}~\medskip

\noindent (1) Let $a,b\in\GG$ with $a\preceq b$ (or $a\preceq_\psi b$). By definition, there exists a (pseudo-)decomposition $a_j$ of $b$ with $a=a_{j_0}$ for some $j_0\ge 1$. Rearranging the first $j_0$ terms, we may always assume $a_1=a$.\medskip

\noindent (2)  The relation $\preceq_\psi$ defines a partial order on $\GG$. More precisely, for $a,b,c\in\GG$, there hold, 
\be\label{poset}
a\preceq_\psi b \text{ and } b\preceq_\psi a\ \iff\ a=b, \qquad\qquad\qquad a\preceq_\psi b \text{ and } b\preceq_\psi c\ \implies\ a\preceq_\psi c.
\ee
As a consequence, $(\GG,\preceq_\psi)$ is a partially ordered set. Similarly $(\SS,\preceq)$ is a partially ordered set.\smallskip

Let us prove~\eqref{poset}. First any $a\in\GG$ admits the pseudo-decomposition $(a,0,0,\dots)$ so $a\preceq_\psi a$. Conversely, let $a,b\in\GG$ such that $a\preceq_\psi b$ and $b\preceq_\psi a$. Let $a_j$ be a pseudo-decomposition of $b$ such that $a_1=a$ and let $b_i$ be a pseudo-decomposition of $a$ such that $b_1=b$.  We compute,
\[
\nu(b)=\nu(a)+\sum_{j\ge 2} \nu(a_j)=\nu(b)+\sum_{i\ge 2} \nu(b_i)+\sum_{j\ge 2} \nu(a_j).
\]
Hence $b_i=a_j=0$ for, $i,j\ge 2$ and $a=b$. This proves the equivalence in the left of~\eqref{poset}.\\
Let us now establish the implication in the right. Let $a,b,c\in\GG$ and assume that $a\preceq_\psi b\preceq_\psi c$. Denoting $a_i$ and $b_j$ some pseudo-decompositions of $b$ and $c$ with $a_1=a$ and $b_1=b$, we have 
\[
c=b+\sum_{j\ge 2}b_j=a+\sum_{i\ge 2}a_i+\sum_{j\ge 2}b_j,
\]
with
\[
\nu(c)=\nu(b)+\sum_{j\ge 2}\nu(b_j)=\nu(a)+\sum_{i\ge 2}\nu(a_i)+\sum_{j\ge 2}\nu(b_j).
\]
Setting $d_1=a$ and then $d_{2j}=a_{j+1}$, $d_{2j+1}=b_{j+1}$ for $j\ge1$, we get that $d_j$ is a pseudo-decomposition of $c$ with $d_1=a$, hence $a\preceq_\psi c$ as claimed.\medskip

\noindent (3.a) If $a\preceq_\psi b$, then $(a,b-a,0,0,\dots)$ is a pseudo-decomposition of $b$. Indeed if $a_j$ is a pseudo-decomposition of $b$ with $a_1=a$, then $b-a=\sum_{j\ge 2}a_j$ and by the triangle inequality,
\[
\nu(a)+\nu(b-a)\ge\nu(b)=\sum\nu(a_j)=\nu(a)+\sum_{j\ge2}\nu(a_j)\ge\nu(a)+\nu(b-a).
\]
Hence $\nu(b)=\nu(a)+\nu(b-a)$ and $(a,b-a,0,0,\dots)$ is a pseudo-decomposition of $b$.\\
If moreover $a$ admits a pseudo-decomposition $c_i$, then the sequence $(b-a,c_1,c_2,\dots)$ is a pseudo-decomposition of $b$. Indeed, we have similarly $b=(b-a)+\sum c_i$ and
\[
\nu(b-a)+\sum\nu(c_i)\ge \nu(b)=\nu(b-a)+\nu(a)=\nu(b-a)+\sum\nu(c_i),
\]
and $\nu(b)=\nu(b-a)+\sum\nu(c_i)$.\medskip

\noindent (3.b) Applying this principle recursively, if  $b_1\succeq_\psi b_2\succeq_\psi\dots$ then for $j\ge 2$, $(b_j,(b_{j-1}-b_j),(b_{j-1}-b_{j-2}),\dots,(b_1-b_2))$ is a pseudo-decomposition of $b_1$, in particular,
\begin{align}
\label{obs_1}
b_1&=b_j+\sum_{1\le i<j}(b_{i+1}-b_i),\\
\label{obs_2}
\nu(b_1)&=\nu(b_j)+\sum_{1\le i<j}\nu(b_{i+1}-b_i).\medskip
\end{align}

\noindent (3.c) By~\eqref{obs_2} we see that the series $\sum(b_j-b_{j+1})$ is absolutely converging and since $\GG$ is complete the sum admits a limit $c_\oo$. In light of~\eqref{obs_1} we see that the sequence $b_j$ also converges and its limit is $b_\oo:=b_1-c_\oo$. Notice that this justifies the existence of this limit which is implicitly assumed in hypothesis (H1). \medskip

\noindent (3.d) Passing to the limit in~\eqref{obs_1}\&\eqref{obs_2} we obtain that
\[
(b_\oo,b_1-b_2,\,b_2-b_3,\,\dots)
\] 
is a pseudo-decomposition of $b_1$.\\
Also remark that  we can start the nonincreasing sequence from any $j>1$ rather than from $1$. We deduce that for $j\ge1$,
\[
(b_\oo,b_j-b_{j+1},\,b_{j+1}-b_{j+2},\,\dots)
\] 
is a pseudo-decomposition of $b_j$.\medskip

\noindent (4) With the same triangle inequality based arguments as above, if $b_j$ is a pseudo-decomposition of $b$ and for each $j$, $(b_{j,i})_i$ is a pseudo-decomposition of $b$ then for any bijection:  
\[
r\in\{1,2,3,\dots\}\mapsto (\ov j(r),\ov i(r))\in\{1,2,3,\dots\}\times\{1,2,3,\dots\},
\] 
the sequence defined by
\[
a_r:=b_{\ov j(r),\ov i(r)}\qquad\text{for }r\ge1,
\]
is a pseudo-decomposition of $b_j$. Besides, if for every $j$, $(b_{j,i})_i$ is a decomposition of $b_j$ (that is $b_{j,i}\in\SS$ for every $i,j$) then $a_r$ is a decomposition of $b$.
\end{observations}

\begin{proof}[Proof of Lemma~\ref{lem_decomp}] In the proof, we use the preceding observations without explicit mention.\medskip
 
\noindent
\textit{Step 1 (Definition and properties of $q(\cd)$).} We define for $b\in\GG$, the quantity,
\[
q(b):=\inf\lt\{\sup \nu(b_j):b_j\text{ decomposition of }b\rt\},
\]
with the convention $q(b)=\oo$ if $b$ does not admit a decomposition. In the other cases, when $b_j$ is a decomposition of $b$, since $\sum\nu(b_j)=\nu(b)<\oo$, the supremum $\sup \nu(b_j)$ is a maximum. \medskip

Let us establish some properties of $q$. First any $b\in\SS$ admits the decomposition $(b,0,0,\dots)$ and we deduce,
\be\label{prf_lem_dec_1}
q(b)\le\nu(b)\qquad\text{for every }b\in\SS.\smallskip
\ee

\noindent First, we claim that if $b_j$ is a pseudo-decomposition of $b$ then,
\be\label{prf_lem_dec_2}
q(b)\le \sup q(b_j).
\ee
To establish~\eqref{prf_lem_dec_2} we assume without loss of generality that $\sup q(b_j)$ is finite. Let $\eps>0$ and for $j\ge 1$, let $(b_{j,1},b_{j,2},\dots)$ be a decomposition of $b_j$ such that $\max_i\nu(b_{j,i})\le q(b_j)+\eps$. Rearranging the countable family $b_{j,i}$, $i,j\ge1$ to form a sequence, we obtain a decomposition $a_r$ of $b$ with $\sup_r\nu(a_r)\le\max_jq(b_j)+\eps$. This yields $q(b)\le\sup_jq(b_j)+\eps$ and then~$q(b)\le \sup q(b_j)$ since $\eps>0$ is arbitrary.\\
Observe that  if $b_j$ is decomposition of $b$, we have $q(b_j)\le \nu(b_j)$ by~\eqref{prf_lem_dec_1} and since $\nu(b_j)\to0$, the supremum in~\eqref{prf_lem_dec_2} is a maximum.\medskip

%

\noindent In the rest of the proof we assume that $b\in\GG$ admits at least one decomposition. Equivalently,
\[q(b)<\oo.\]
We now establish the following.
\be\label{prf_lem_dec_3}
\forall\,\eps>0\ \exists\, b_1\preceq b\text{ such that}
\lt\{\begin{array}{l}
\phantom{aaa_1}q(b)\le q(b_1)\le\nu(b_1)\le q(b)+\eps,\\
q(b-b_1)\le q(b_1). 
\end{array}
\rt.
\ee
Let $b_j$ be a decomposition of $b$ with $\max\nu(b_j)\le q(b)+\eps$. Since $q(b_j)\le\nu(b_j)$ and $\sum \nu(b_j)<\oo$, the sequence $q(b_j)$ reaches its maximum. Reordering if necessary, we assume $q(b_1)=\max q(b_j)$. There holds, 
\[
q(b)\st{\eqref{prf_lem_dec_2}}\le \max q(b_j) =q(b_1)\st{\eqref{prf_lem_dec_1}}\le\nu(b_1)\le q(b)+\eps.
\]
 Moreover, by~\eqref{prf_lem_dec_2} again applied to the decomposition $(b_j)_{j\ge 2}$ of $b-b_1$, we have $q(b-b_1)\le\max_{j\ge 2}q(b_j) \le q(b_1)$. This proves~\eqref{prf_lem_dec_3}.\medskip

\noindent
\textit{Step 2 (Extraction of a ``big atom'').}  Let $b\in\GG$ with $q(b)<\oo$. Let us establish that:
\be\label{prf_lem_dec_4}
\text{there exists an atom }a\preceq b\text{ such that }q(b-a)\le\nu(a).
\ee
Let $\eps_j>0$ with $\eps_j\to0$. We build recursively $b_0\succeq b_1\succeq b_2\succeq\dots$ by setting $b_0=b$ and then by applying~\eqref{prf_lem_dec_3} to $b_j$ with $\eps=\eps_j$. We have for $j\ge 0$,
\begin{align}
\label{prf_lem_dec_5}
q(b_j)&\le q(b_{j+1}) \le \nu(b_{j+1} ) \le q(b_j) +\eps_j,\\ 
\label{prf_lem_dec_6}
q(b_j-b_{j+1})&\le q(b_{j+1}).
\end{align}

Applying~(H1) to the nonincreasing sequence $b_j$, there exists $a\in\SS$ such that $b_j\to a$ and by Observations~\ref{obs_}(3.d)\&(4) we have $a\preceq b_j$ for every $j\ge0$ and in particular $a\preceq b=b_0$. Moreover, from~\eqref{prf_lem_dec_5}, $q(b_j)$ is nondecreasing and
\be
\label{prf_lem_dec_7}
\sup_{j\ge0}q(b_j)=\lim_{j\up\oo}q(b_j)= 
\lim_{j\up\oo}\nu(b_j)=\inf_{j\ge0}\nu(b_j)=\nu(a).
\ee
Next, for $j\ge1$, the sequence $(b_i-b_{i+1})_{i\ge j}$ is a pseudo-decomposition of $b_j-a$ and we deduce from~\eqref{prf_lem_dec_2} that
\be
\label{prf_lem_dec_8}
q(b_j-a)\le\max_{i\ge j} q(b_i-b_{i+1}).
\ee
Applying this inequality with $j=0$, we get (recall $b_0=b$),
\[
q(b-a)\le\max_{i\ge 0}q(b_i-b_{i+1})\st{\eqref{prf_lem_dec_6}}\le\sup_{i\ge 0}q(b_{i+1})\st{\eqref{prf_lem_dec_7}}= \nu(a).
\] 
This proves the inequality in~\eqref{prf_lem_dec_4}.

We still have to check that $a$ is an atom. For this, we first prove:
  \be\label{prf_lem_dec_9}
 \lim_{j\up\oo}q(b_j-a)=0.
 \ee
 Combining~\eqref{prf_lem_dec_8} and ~\eqref{prf_lem_dec_2} we have, $q(b_j-a)\le \max_{i\ge j} \nu(b_i-b_{i+1})$ and since  $\sum_{i\ge 0}  \nu(b_i-b_{i+1})<\oo$, the right-hand side goes to $0$ as $j\up\oo$. This proves~\eqref{prf_lem_dec_9}.

Let now  $(c_i)$ be a decomposition of $a$ and assume without loss of generality that $\max\nu(c_i)=\nu(c_1)$. For $j\ge 0$, the sequence $(b_j-a,c_1,c_2,\dots)$ is a pseudo-decomposition of $b_j$. We compute,
\[
q(b_j)
\st{\eqref{prf_lem_dec_2}\eqref{prf_lem_dec_1}}
\le\max\Big(q(b_j-a),\,\max_i\nu(c_i)\Big)
=\max\lt(q(b_j-a),\,\nu(c_1)\rt).
\]
Sending $j$ to $\oo$, by~\eqref{prf_lem_dec_7}, the left-hand side converges towards $\nu(a)=\sum\nu(c_i)$ and, by~\eqref{prf_lem_dec_9}, the right-hand side towards $\nu(c_1)$. We obtain,
\[
\sum \nu(c_i)\le\nu(c_1),
\]
hence $c_i=0$ for $i\ge 2$ and $a$ is an atom. The claim~\eqref{prf_lem_dec_4} is established.
\medskip

\noindent
\textit{Step 3 (Conclusion).} Let $b\in\GG$ such that $q(b)<\oo$. We build recursively two sequences $a_j\preceq b$, $b_j\preceq_\psi b$ indexed by $j\ge0$ such that the $a_j$'s are atoms and for $j\ge 1$, $(a_1,a_2,\dots,a_j,b_j,0,\dots)$ is a pseudo-decomposition of $b$.  For this, we start with  $a_0:=0$, $b_0:=b$ and then for $j\ge 0$ we apply~\eqref{prf_lem_dec_4} to $b_j$ to get an atom $a_{j+1}\preceq b_j$. We then set $b_{j+1}:=b_j-a_{j+1}$ so that $(a_{j+1},b_{j+1})$ is a pseudo-decomposition of $b_j$ and proceed to the next step.  By construction, $(a_1,a_2,\dots,a_{j+1},b_{j+1},0,\dots)$ is a pseudo-decomposition of $b$. Moreover, by~\eqref{prf_lem_dec_4}, for $j\ge0$
\be\label{prf_lem_dec_10}
q(b_{j+1})\le \nu(a_{j+1}).
\ee
In particular $q(b_{j+1})<\oo$ and we can apply~\eqref{prf_lem_dec_4} to $b_{j+1}$ and continue the construction.\\
The pseudo-nonincreasing sequence $b_j$ converges to some $b_\oo\in\GG$  (because $\GG$ is complete) 
 and $(a_1,a_2,\dots)$  is a decomposition of $b-b_\oo=b_0-b_\oo$ in atoms.\smallskip

To conclude, we establish that $b_\oo=0$. Let us fix $j\ge 1$ and let $(b_{j,1},b_{j,2},\dots)$ be a decomposition of $b_j$ such that $\nu(b_{j,i})\le2q(b_j)$ for $i\ge1$. 
Using the triangle inequality and (H2) we have for $k\ge 1$, 
\be\label{interm}
\phi(b_j)\le\sum_{i=1}^k \phi(b_{j,i}) +\phi\lt(\sum_{i>k} b_{j,i}\rt) \le\sum_{i\ge1}\phi(b_{j,i}) +\phi\lt(\sum_{i>k} b_{j,i}\rt).
\ee
Now, by (H2), for any sequence $c_j\in \GG$, there holds, 
\be\label{phi<nu}
\nu(c_j)\to0\ \implies\ \phi(c_j)\to0.
\ee
Since $\sum b_{j,i}$ converges in $\nu$-norm, we deduce that the last term of~\eqref{interm} goes to 0 as $k\to+\oo$ and we  get,
\[
\phi(b_j)\le\sum_{i\ge1} \phi(b_{j,i}).
\] 
Then, applying~(H2) to the $b_{j,i}$'s, we compute
\[
\phi(b_j)\le\sum_{i\ge1}\eta(\nu(b_{j,i}))\,\nu(b_{j,i}) \le \eta(2q(b_j))\sum_{i\ge1}\nu(b_j^i) = 
\eta(2q(b_j))\,\nu(b_j)\le
\eta(2q(b_j))\,\nu(b).
\]
Using~\eqref{prf_lem_dec_10}, we obtain,
\[
\phi(b_j)\le\eta(2\nu(a_j))\,\nu(b).
\]
Since $\sum\nu(a_j)\le\nu(b)<\oo$ we have $\nu(a_j)\to0$, hence $\eta(2\nu(a_j))\to0$ and we obtain $\phi(b_j)\to0$. Eventually, by triangle inequality, we have
\[
\phi(b_\oo)\le\phi(b_j)+\phi(b_j-b_\oo)\ \st{j\up\oo}\longto\ 0,
\]
where we apply~\eqref{phi<nu} with $c_j=b_j-b_\oo$ to see that the last term goes to 0. 
Since $\phi$ is a norm we conclude that $b_\oo=0$ which proves the lemma.
\end{proof}
\begin{remark}
It transpires from the proof that we do not need the whole ``normed group'' structure of $\GG$. Indeed, the inequality $\nu(c-a)\le\nu(c)+\nu(a)$ is never used. By inspecting the demonstration, we see that the lemma is still correct if $(\GG,+,\nu)$ is a complete normed commutative \emph{monoid}.  In this setting, the identities of the form $b=c-a$ in the proof should be read as $a+b=c$. For instance, the term $b-a$ of Observations~\ref{obs_}~(3.a) should be defined as $\sum_{j\ge2} a_j$ where $(a,a_2,a_3,\dots)$ is a pseudo-decomposition of $b$.
\end{remark}

\section{Proof of Theorem~\ref{thm_decomp}}
\label{S_proof_td}
As in the introduction, $(G,+,|\cdot|_G)$ is a complete Abelian normed group and $0\le k\le n$. Let us first establish an isoperimetric inequality for normal rectifiable $G$-flat chains. 
In light of Almgren's isoperimetric inequality~\cite{Almgren,White1999-1}, the result is not that surprising and might not be new. Its proof is based on the deformation theorem of White~\cite{White1999-1} 
and follows the steps of the proof of Federer's isoperimetric inequality for integral currents.\\
Let us recall the definition of the $h$-mass of a rectifiable chain, see~\cite[Section~6]{White1999-1}.
\begin{definition}~\label{def_hmass}

Let $h:\R_+\to\R_+$ be a lower semicontinuous and subadditive function satisfying $h(0)=0$. 
The $h$-mass of a rectifiable $k$-chain $A= w \xi \h^k\restr \Sigma$ is defined as
\[
\M_h(A):=\int_\Sigma h(\rho)\, d\h^k, 
\]
where $\rho:=|w|_G$.
\end{definition}
The condition $h(0)=0$ ensures that the definition does not depend on the choice of $(\Sigma,w)$. The  lower-semicontinuity and subadditivity properties are necessary to get good properties of $\M_h$ with respect to  convergence and projections/deformations. Let us recall in particular that under these assumptions $\M_h$ is countably subadditive (see~\cite[Section~6]{White1999-1}), that is,
\be\label{count_sub}
\M_h\lt(\sum B_j\rt)\le\sum\M_h(B_j)\qquad\text{for }B_j\in\MM^G_k(\R^n)\text{, rectifiable.}
\ee
Of course, the case $h(s)=s$ corresponds to the usual mass.

\begin{lemma}[Isoperimetric inequality for normal rectifiable chains]~\\
\label{lem_isopineq}
Let $h:\R_+\to\R_+$ be lower semicontinous, subadditive and such that $h(0)=0$, $h>0$ on $(0,+\oo)$ and  $h'(0^+):=\lim_{s\dw0}h(s)/s=\oo$.
There exists a nondecreasing function $\eta:\R_+\to\R_+$ only depending on $n$ and $h$ such that $\lim_{m\dw0} \eta(m)=0$ and:
\[
\F(A)\le\eta(\M(A))\,(\M_h(A)+\N(A))
\]
for every normal rectifiable chain $A\in\NN_k^G(\R^n)$.
\end{lemma}
\begin{proof}
Let $h$ and $A$ as in the statement of the lemma and let $\eps>0$ to be fixed later. 
Applying the  deformation theorem~\cite[Theorem~1.1]{White1999-1}, there exist a constant $c\ge1$ only depending on $n$ and chains $P\in\PP_k^G(\R^n)$, $R\in\FF_k^G(\R^n)$ and $S\in\FF_{k+1}^G(\R^n)$ with the following properties:
\begin{enumerate}[(a)]
\item $A=P+R+\pt S$.
\item $P=\sum g_F F$, with $g_F\in G$ and where we sum over a countable set of essentially disjoint oriented $k$-cubes $F$ with side length $\eps$.  
\item[(c)]$\M(P)\le c\M(A)$.
\item[(d)]$\M_h(P)\le c\M_h(A)$.
\item[(e)]$\M(R)+\M(S)\le c\eps \N(A)$.
\end{enumerate}
 Let us denote $\nu(A):=\M_h(A)+\N(A)$. By~(a) and~(e) there holds
\be\label{prf_lem_isop_1}
\F(A)\le\M(P)+c\eps\nu(A). 
\ee
To estimate the first term, we write, using the formula~(b) and the inequality~(c),
\begin{equation}\label{massP}
\M(P)=\eps^k\sum_{F}|g_F|_G\le c\M(A).
\end{equation}
We deduce 
\be\label{prf_lem_isop_2}
\max_F |g_F|_G\le c\M(A)\eps^{-k}.
\ee 
We now set $\wt\eta(0):=0$ and for $m>0$, 
\[
\wt\eta(m):=\sup\lt\{\dfrac s{h(s)}:0<s\le m\rt\}.
\] 
From the assumption on $h$, $\wt\eta$ is nondecreasing and $\lim_{m\dw0}\wt\eta(m)=0$. We compute, using~(b),
\begin{multline*}
\M(P)
=\eps^k\sum_F|g_F|_G
\le\eps^k\sum_F\wt\eta(|g_F|_G) h(|g_F|_G)
\st{\eqref{prf_lem_isop_2}}\le\wt\eta\lt(c\M(A)\eps^{-k}\rt)\eps^k\sum_Fh(|g_F|_G)\\
=\wt\eta\lt(c\M(A)\eps^{-k}\rt)\M_h(P)\st{\text{(d)}}\le c\wt\eta\lt(c\M(A)\eps^{-k}\rt)\nu(A).
\end{multline*}
Putting this estimate in~\eqref{prf_lem_isop_1}, we get
\[
\F(A)\le c \lt[\wt\eta\lt(c\M(A)\eps^{-k}\rt)+\eps\rt]\nu(A).
\]
Then, taking the infimum over $\eps>0$ and setting for $m>0$,
\begin{equation}\label{prf_lem_isop_3}
\eta(m):=c\inf_{\eps>0} \lt[\wt\eta\lt(cm\eps^{-k}\rt)+\eps\rt],
\end{equation}
we obtain
\[
\F(A)\le\eta(\M(A))\nu(A)=\eta(\M(A))(\M_h(A)+\N(A)).
\]
Eventually, since $\wt\eta$ is nondecreasing, $\eta$ is nondecreasing. Moreover, given $\d>0$ we first fix $\eps:=\d/(2c)$, then using $\wt\eta(s)\to0$ as $s\dw0$ we have $\wt\eta(cm\eps^{-k})<\d/(2c)$ for $m\ge0$ small enough. We deduce that for such $m$'s, $\eta(m)<\d$. This establishes that $\eta(m)\to0$ as $m\dw0$ and ends the proof of the lemma.
\end{proof}
\begin{remark}\label{rem_isoper}
Assuming that $h$ is increasing and using $\M_h$ instead of $\M$ in~\eqref{massP} we can replace~\eqref{prf_lem_isop_2} by
 \[
  \max_F |g_F|_G\le  h^{-1}(c\M_h(A)\eps^{-k}).
 \]
Therefore, setting $\wt\eta_*(m):=h^{-1}(m)/m$ and then defining $\eta_*$ by~\eqref{prf_lem_isop_3} with $\wt\eta_*$ instead of $\wt\eta$ we have also the inequality
\[
 \F(A)\le\eta_*(\M_h(A))\,(\M_h(A)+\N(A)).
\]
In the particular case $h(m)=m^\alpha$ for $\alpha\in(0,1)$ this yields
\[
\F(A)\le c \M_h(A)^{\frac{1-\alpha}{\alpha+k(1-\alpha)}}(\M_h(A)+\N(A)).
\]
Sending $\alpha$ to $0$, so that in the limit $\M_h(A)=\textrm{Size}(A)$ we recover Almgren's isoperimetric inequality~\cite{Almgren} (see also~\cite[Theorem 6.2]{White1999-1}).
\end{remark}

We can now prove the main result.
\begin{proof}[Proof of Theorem~\ref{thm_decomp}]
Let $A\in\NN_k^G(\R^n)$ be a normal rectifiable $k$-chain.\medskip

\noindent
\textit{Step 1.} We first claim that: 
\[
\text{There exists }h\in C(\R_+,\R_+)\text{ increasing, concave and such that }
\lt\{
\begin{array}{l}
h(0)=0,\\ 
h'(0^+)=\oo,\\ 
\M_h(A)<\oo.
\end{array} 
\rt.
\]
With the notation of Definition~\ref{def_hmass}, $\M(A)=\int\rho\,d\h^k\restr\Sigma<\oo$ and for $h\in C(\R_+,\R_+)$,
\[
\M_h(A)=\int h(\rho)\,d\h^k\restr\Sigma.
\]
The claim is then a direct application of Lemma~\ref{lem_DLVP} established in appendix, with the nonnegative function $f= \rho$ and the measure $\mu=\h^k\restr\Sigma$.\\
Remark that $h$ is strictly subadditive so that it satisfies the assumptions of Lemma~\ref{lem_isopineq}.
\medskip

\noindent
\textit{Step 2.} For $B\in\FF_k(\R^n,G)$, we set 
\be\label{def_nu}
\nu(B):=\M_h(B)+\N(B),
\ee
and define
\be\label{def_G}
\GG:=\lt\{B\in\FF_k(\R^n,G)\text{, rectifiable and such that }\nu(B)<\oo\rt\}.
\ee
The mapping $\nu$ is obviously a norm on $\GG$ (recall~\eqref{count_sub}).  Let us show that $(\GG,\nu)$ is complete. Let $B_j$ be a Cauchy sequence in $(\GG,\nu)$. In particular, $B_j$ is a Cauchy sequence in $(\NN^G_k(\R^n),\N)$ which is complete and there exists $B_\oo\in\NN^G_k(\R^n)$ such that $\N(B_j-B_\oo)\to0$.\\
Now, there exists a $k$-rectifiable set $\Sigma\sub\R^n$ such that $B_j=B_j\restr \Sigma$ for every $j\ge1$. Using transparent notation, we write $B_j=w_j\h^k\restr \Sigma$ and we observe that the convergence $\M(B_j-B_\oo)\to0$ rewrites as 
\be\label{prf_thm_dec_1}
w_j\to w_\oo\quad\text{in }L^1\lt((\Sigma,\h^k),(G,|\cd|_G)\rt).
\ee
Similarly, denoting $|g|^*_G:=h(|g_G|)$ for $g\in G$, the fact that $B_j$ is a Cauchy sequence with respect to $\M_h$ rewrites as
\[
w_j\text{ is a Cauchy sequence in }L^1\lt((\Sigma,\h^k),(G,|\cd|^*_G)\rt).
\]
The latter group being complete, there exists $w^*_\oo$ such that
\[
w_j\to w^*_\oo\quad\text{in }L^1\lt((\Sigma,\h^k),(G,|\cd|^*_G)\rt).
\]
Besides, with~\eqref{prf_thm_dec_1} we have $w^*_\oo=w_\oo$ $\h^k$-almost everywhere on $\Sigma$ and we see that $\M_h(B_j-B_\oo)\to0$. This proves that $(\GG,\nu)$ is complete.\medskip

Now we set
\be\label{def_S}
\SS:=\lt\{A\restr S : S\sub \R^n\text{ Borel set such that }\N(A\restr S)<\oo\rt\}.
\ee
Notice that by Remark~\ref{rem:rectif} every element of $\SS$ is rectifiable.\\
We claim that for $B\in\GG$ and $A_1,A_2,\cdots\in \SS$ we have (using the notation of Lemma~\ref{lem_decomp} with~\eqref{def_nu}\eqref{def_G}\&\eqref{def_S}),
\be\label{prf_thm_dec_2}
A_j\text{ decomposition of }B\text{ as in Lemma~\ref{lem_decomp}} \quad\iff\quad A_j\text{ set-decomposition of }B.
\ee 

Assuming that $(A_1,A_2,\dots)$ is a set-decomposition of $B$, we write $A_j=B\restr S_j$ where $S_j$ is a Borel partition of $\R^n$. We have obviously $\M_h(B)=\sum \M_h(A_j)$ and by definition $\N(B)=\sum \N(A_j)$, hence $(A_1,A_2,\dots)$ is a decomposition of $B$ in the sense of Lemma~\ref{lem_decomp}.\medskip

\noindent
Conversely, if $(A_1,A_2,\cdots)$ is a decomposition of $B$ in the sense of Lemma~\ref{lem_decomp}, then $B=\sum A_j$ and $\nu(B)=\sum\nu(A_j)$. Since, by the triangle inequality,
\[
\lambda(B)\le\sum \lambda(A_j)\qquad\text{for } \lambda=\M,\ \M_h\text{ and }\M(\pt \cdot),
\]
the identity $\nu(B)=\sum\nu(A_j)$ yields
\[
\lambda(B)=\sum \lambda(A_j)\qquad\text{for } \lambda=\M,\ \M_h\text{ and }\M(\pt \cdot).
\]
Hence $\N(B)=\sum\N(A_j)$.\\
It remains to check that the $A_j$'s belong to $\SS$. With the notation of Definition~\ref{def_hmass}, we set 
\[
\mu_A:=\rho\h^k\restr\Sigma,
\] 
so that $\M(A\restr S)=\mu_A(S)$ for every Borel set $S\sub\R^n$.\\
Writing $B=A\restr S$ and similarly $A_j=A\restr S_j$ for $j\ge1$ for some Borel subsets $S,S_j\sub\R^n$, the convergence of  $\sum A_j$ towards $B$ in mass is equivalent to $\sum \un_{S_j}\to\un_S$ in $L^1(\R^n,\mu_A)$. We deduce that, up to $\mu_A$-negligible sets, $S_j$ is a partition of $S$. This proves~\eqref{prf_thm_dec_2}.\medskip

The proof that  $B$ is an atom in the sense of Lemma~\ref{lem_decomp} if and only if $B$ is set-indecomposable is similar.\medskip

\noindent 
\textit{Step 3.} By the previous step the theorem follows from Lemma~\ref{lem_decomp} (and Remark~\ref{rem_lem_decomp}~(1)) applied to $(\GG,+,\nu)$ and $\SS$ provided that~(H1)\&(H2) hold true.\medskip

First, choosing the norm $\phi=\F$, Assumption~(H2) is a direct consequence of Lemma~\ref{lem_isopineq}.\medskip

Let us establish~(H1). Let $B_1\succeq B_2\succeq\dots$ be a nonincreasing sequence. Referring to Observations~\ref{obs_}~(3.c), there exists $B_\oo\in\GG$ such that $\nu(B_j-B_\oo)\to0$. Next, reasoning as above, there exists a nonincreasing sequence $S_j$ of Borel subsets of $\R^n$ such that $B_j=A\restr S_j$. Defining $S_\oo:=\cap S_j$  and $B^*_\oo:=A\restr S_\oo$, we have by the monotone convergence theorem,
\[
\lim\M(B_j-B^*_\oo)=0.
\]
Consequently, $B_\oo=B^*_\oo=A\restr S_\oo$ and $B_\oo\in\SS$. This proves~(H1).\\
As a conclusion the theorem follows from Lemma~\ref{lem_decomp}.
\end{proof}


\section{Uniqueness of the set-decomposition of normal $n$-chains with real coefficients}
\label{S_uniq_n}
As announced in the introduction, we are able to establish the uniqueness of the maximal set-decomposition of normal codimension $0$ chains when $(G,+,|\cd|_G)$ is a closed subgroup of $(\R,+,|\cd|)$.
\begin{proposition}\label{prop_uniq_n}
If $G$ is the additive group $\R$ endowed with the standard norm, then the decomposition of a normal $G$-flat $n$-chain in $\R^n$ in set-indecomposable subchains is unique up to rearranging the sequence and adding or deleting zeros.    
\end{proposition}
The result also applies to $G=\Z$ by the embedding $\NN_n^\Z(\R^n)\hookrightarrow\NN_n^\R(\R^n)$. 
The proof of the proposition is based  on the coarea formula for functions of bounded variations and on the uniqueness of the decomposition of a set of finite perimeter in its measure theoretic connected components provided by~\cite[Theorem~1]{ACMM}. We first reformulate the proposition as a result about functions of bounded variation (Theorem~\ref{thm_dec_BV} below). 

It is well-known that the space of $\R$-flat $n$-chains in $\R^n$ in the sense of~\cite{Fleming66} identifies with a subspace of $k$-currents in $\R^n$, namely, the closure of the space of normal $n$-currents with respect to the norm
\[
W(T):=\sup\lt<T,\om\rt>,
\] 
where the supremum is taken over the smooth and compactly supported differential $n$-forms $\om$ over $\R^n$ such that $\|\om\|_\oo\le1$. This space obviously identifies isometrically with $L^1(\R^n)$ and denoting $f_A$ the function corresponding to a $\R$-flat $n$-chain $A$, we have $\F(A)=\M(A)=\|f_A\|_{L^1}$ and $\pt A$ is the $(n-1)$-current $\sum_{i=1}^n \pt_{x_i} f_A\h^n\, e_{\ov i}$ where  $e_1,\dots,e_n$ is the standard basis of $\R^n$ and 
\[
e_{\ov i}:= e_1\we\dots\we e_{i-1}\we e_{i+1}\we\dots\we e_n.
\]  
Using the Hodge star operator $e_{\ov i}\mapsto e_i$, $\pt A$ identifies with the distribution $\nb f_A$. Moreover, the $n$-current $A$ is normal if and only if $f_A$ is a function with bounded variation and we have the identity $\M(\pt A)=|\nb f_A|_{TV}$ where here and below the total variation of  a vector valued Borel measure $\mu\in\MM(\R^n,\R^d)$ is computed with respect to the Euclidian norm in $\R^d$, that is, 
\[
|\mu|_{TV}:=\sup\lt\{\sum_{S_j} \lt\|\mu(S_j)\rt\|_{\ell^2(\R^d)}: S_j\text{ Borel partition of }\R^n\rt\}.
\]
Next, given a Borel set $S\sub\R^n$, we have $f_{A\restr S}=\un _S f_A$ and a set-decomposition of $A$ corresponds to a finite or countable Borel partition $\BB$ of $\R^n$ such that 
\be
\label{nbfA_decomp}
|\nb f_A|_{TV}=\sum_{S\in\BB}\lt|\nb [\un_{S} f_A]\rt|_{TV}.  
\ee
Denoting $\Om:=\{x\in\R^n:f_A(x)\ne0\}$, we have $f_A=\un_\Om f_A$ so we may only consider Borel partitions of $\Om$.  A set-decomposition corresponds to a Borel partition of  $\Om$ satisfying~\eqref{nbfA_decomp} and $A$ is indecomposable if for every Borel set $S\sub\Om$ there holds
\[
\lt|\nb f_A\rt|_{TV}=\lt|\nb [\un_Sf_A]\rt|_{TV}+\lt|\nb[\un_{\R^n\sm S}f_A]\rt|_{TV}\quad\implies\quad \h^n(S)=0\text{ or }\h^n(\Om\sm S)=0.
\]
Let us state in terms of $BV$-functions both the existence result of Theorem~\ref{thm_decomp} in the case $k=n$, $G=\R$ and the uniqueness result (still to be proved) of Proposition~\ref{prop_uniq_n}.
\begin{theorem}\label{thm_dec_BV}
Let $f\in BV(\R^n)$, there exists a Borel partition $\BB$  of $\Om:=\lt\{x\in\R^n:f(x)\ne0\rt\}$ such that 
\be\label{f_decomp}
\lt|\nb f\rt|_{TV}=\sum_{S\in\BB}\lt|\nb [\un_Sf]\rt|_{TV}
\ee
and such that for any Borel partition $\BB'$ with the same properties and any $S\in\BB$ we have $\h^n(S\sm S')=0$ for some $S'$ in $\BB'$. In other words, $\BB$ is the finest Borel partition of $\Om$ satisfying~\eqref{f_decomp}.
\end{theorem}
\begin{proof}[Proof of Theorem~\ref{thm_dec_BV}]~\\
\textit{Step 0 (conventions).} In this proof we  identify Borel subsets of $\Om$ which only differ by a Lebesgue null set and we make an abuse of notation by writing $S\subset S'$ if $S\sm S'$ is a null set. With this convention, given $\BB$, $\BB'$ two families of Borel subsets of $\Om$, we write $\BB\Subset\BB'$ whenever 
\[
\text{for every element } S\in\BB \text{ there exists }S'\in\BB'\text{ such that }S\subset S'.
\]
This defines a partial order on the families of Borel subsets of $\R^n$. The theorem states that the collection of Borel partitions of $\Om$ satisfying~\eqref{f_decomp} admits a least element for the relation $\Subset$. Similarly, with this vocabulary,~\cite[Theorem~1]{ACMM} states that a set of finite perimeter $E\sub\R^n$ admits a Borel partition $\BB_E$ whose elements are called the M-connected components of $E$ such that 
\[
P(E)=\sum_{F\sub \BB_E} P(F)
\]
and such that if $\BB_E'$ is any other Borel partition of $E$, 
\be\label{prf_uniq_0}
P(E)=\sum_{F\sub \BB_E'} P(F)\quad\iff\quad\BB_E\Subset\BB_E'.\smallskip
\ee

\noindent
\textit{Step 1.}  Let $f\in BV(\R^n)$. For $t\in\R\sm\{0\}$, we set 
\[
E_t:=\begin{cases} 
\lt\{x\in\R^n :f(x)>t\rt\}&\text{if }t>0,\\
\lt\{x\in\R^n :f(x)<t\rt\}&\text{if }t<0.
\end{cases}
\]
For almost every $t$, $E_t$ is a set of finite perimeter and denoting by $P(E)$ the perimeter of $E\sub\R^n$ (that is the $\h^{n-1}$-measure of the reduced boundary of $E$), the mapping $t\mapsto P(E_t)$ is measurable and by the coarea formula~\cite[Theorem~3.40]{Am_Fu_Pal},
\be\label{prf_uniq_1}
\lt|\nb f\rt|_{TV}=\int_\R P(E_t)\, dt.
\ee
Le us denote by $\BB_t$ the collection of the M-connected components of $E_t$ given by~\cite[Theorem~1]{ACMM}. For almost every $t\in \R$, $\BB_t$ is a finite or countable Borel partition of $E_t$ and 
\[
P(E_t)=\sum_{F\in\BB_t}P(F).\smallskip
\]
Let $\BB$ be a Borel partition of $\Om$ such that~\eqref{f_decomp} holds true. Since for $t\in\R\sm\{0\}$ and $S\in\BB$,
\[
E_t\cap S=\begin{cases} 
\lt\{x\in\R^n :\un_S(x)f(x)>t\rt\}&\text{if }t>0,\\
\lt\{x\in\R^n :\un_S(x)f(x)<t\rt\}&\text{if }t<0,
\end{cases}
\]
we have again by the coarea formula,
\[
\lt|\nb[\un_Sf]\rt|_{TV}=\int P(E_t\cap S)\, dt.
\]
With~\eqref{f_decomp} and~\eqref{prf_uniq_1} this leads to 
\[
\int_\R P(E_t)\, dt = \int_\R \sum_{S\in\BB}P(E_t\cap S)\, dt.
\]
Since for every $t$, $E_t=\cup_S(E_t\cap S)$ we have $P(E_t)\le\sum_SP(E_t\cap S)$, the preceding inequality enforces 
\[
P(E_t)=\sum_{S\in\BB}P(E_t\cap S)\qquad\text{for almost every }t\in\R\sm\{0\}. 
\]
By~\eqref{prf_uniq_0} this implies $\BB_t\Subset \{E_t\cap S:S\in\BB\}\Subset\BB$ and we conclude that 
\be\label{prf_uniq_2}
\text{for almost every }t\in\R\sm\{0\}\quad\BB_t\Subset\BB.\smallskip
\ee
\noindent
\textit{Step 2.} Now we claim that the collection $\BB_t$ for $t\ne0$ admits a $\Subset$-maximal element $\BB_0$. Indeed, for $x,y\in\Om$ let us write $x\sim y$ whenever there exists $t=t_{x,y}\ne0$ such that $E_t$ is a set of finite perimeter and $x$ and $y$ are both points of density of the same $F\in\BB_t$. Remark that if $t_1<0<t_2$ then for $S_1\in\BB_{t_1}$, $S_2\in\BB_{t_2}$ there holds $S_1\cap S_2=\void$ and  if $0<t_1<t_2$ or $t_2<t_1<0$ then $\BB_{t_2}\Subset\BB_{t_1}$. We deduce that $\sim$ defines an equivalence relation on the set
\[
\Om_0:=\bigcup_{t\ne0}\bigcup_{F\in\BB_t}\lt\{x\text{ point of density of }F\rt\}
\]
which is of full measure in $\Om$. Moreover we can impose that the $t_{x,y}$'s above lie in a countable set $T$ such that $\sup T\cap(-\oo,0)=\inf T\cap(0,+\oo)=0$. It follows that each equivalence class of $\Om_0/\!\!\sim$ writes as countable union of sets with finite perimeter and in particular, up to negligible sets, $\BB_0:=\Om_0/\!\!\sim$ is a Borel partition of $\Om_0$.

By construction and~\eqref{prf_uniq_2} we have $\BB_0\Subset\BB$ for any Borel partition $\BB$ of $\Om$ satisfying~\eqref{f_decomp}. To end the proof of the theorem we still have to check that $\BB=\BB_0$ satisfies~\eqref{f_decomp}.\\
For $t\ne0$ and $S\in\BB_0$, we define
\[
S_t:=S\cap E_t=\begin{cases} 
\lt\{x\in\R^n :\un_S(x)f(x)>t\rt\}&\text{if }t>0,\\
\lt\{x\in\R^n :\un_S(x)f(x)<t\rt\}&\text{if }t<0,
\end{cases}
\]
On the one hand, we have for almost every $t$ and any $S\in\BB_0$,
\be\label{prf_uniq_3}
\lt|\nb[\un_Sf]\rt|_{TV} = \int_\R P(S_t)\, dt.
\ee
On the other hand, defining $\BB'_t:=\{S_t:S\in\BB_0\}$ we have by definition of $\BB_0$ that for almost every $t\in\R\sm\{0\}$, $\BB'_t$ is a partition of $E_t$ with $\BB_t\Subset\BB'_t$. We deduce from~\eqref{prf_uniq_0} that
\[
P(E_t)=\sum_{S'\in\BB'_t}P(S')=\sum_{S\in\BB_0}P(S_t). 
\]
Integrating over $t\in\R$ and using~\eqref{prf_uniq_3} we get
\[
\lt|\nb f\rt|_{TV}=\sum_{S\in\BB_0}\int P(S_t)\, dt=\sum_{S\in\BB_0}\lt|\nb[\un_Sf]\rt|_{TV}.
\] 
So $\BB=\BB_0$ satisfies~\eqref{f_decomp} and the theorem is established.
\end{proof}
We believe that Proposition~\ref{prop_uniq_n} generalizes to any Abelian normed groups $G$ but addressing the general case would take us too far afield. An idea would be to identify the normal $n$-chain $A$ with a $G$-valued function $f_A:\R^n\to G$ such that $f_A$ is of bounded variation in the sense of~\cite{Ambrosio1990}. Such identification is established in~\cite[Theorem~4.1]{HdP2014} but to complete the program we need the identity 
\be\label{MTV}
\M(\pt A)=|\nb f_A|_{TV}:=\sup \sum_j |\nb [\phi_j\circ f_A]|(S_j),
\ee
where the supremum is taken over the countable Borel partitions $S_j$ of $\R^n$ and over the sequences of  1-Lipschitz continuous functions $\phi_j:G\to\R$. Unfortunately, the result of~\cite{HdP2014} only provides a bilipschitz group isomorphism $A\in\NN_n^G(\R^n)\to BV(\R^n,G)$ and we only have at hand a two-sided estimate in place of the identity~\eqref{MTV}.

\appendix
\section{Higher integrability lemma}
\label{App1}
\begin{lemma}\label{lem_DLVP} 
Let $(\Om,\mu)$ be a measure space. Given a $\mu$-integrable function $f:\Om\to\R_+$, there exists $h:\R_+\to\R_+$, continuous, increasing, concave such that $h(0)=0$, $h'(0^+):=\lim_{s\dw 0} h(s)/s=\oo$ and 
\[
\int h(f)\,d\mu<\oo.
\]
\end{lemma}
\begin{proof}
This is the consequence of the following simple higher summability property for absolutely converging series. Namely, if $a_j\ge0$ is such that $\sum a_j<\oo$ then there exists a sequence $1=b_0<b_1<\dots <b_j<b_{j+1}<\dots\to\oo$ such that $\sum a_j b_j<\oo$.\footnote{This is easy to establish. Build an increasing sequence of integers $m_1,m_2,m_3,\dots$ such that $\sum_{i>m_l}a_i\le 2^{-l}$ for $l\ge1$ then set $b_0:=1$, $b_{m_l}:=l$ for $l\ge 1$ and complete the definition of the $b_i$'s  by affine interpolation.} Applying this to the series 
\[
\sum_{j\ge0} \int_{\{2^{-j-1}<f \le 2^{-j}\}} f\,d\mu=\int f\,d\mu<\oo,
\]
we get $1=b_0<b_1<\dots<b_j<b_{j+1}<\dots\to\oo$ such that,
\[
\sum_{j\ge0}b_j \int_{\{2^{-j-1}<f \le 2^{-j}\}} f\,d\mu<\oo.
\] 
Defining $c_0:=1$ and then recursively $c_j:=\min\lt(\sqrt2\,,{b_j}/{b_{j-1}}\rt)c_{j-1}$ for $j\ge1$,
we have $1=c_0<\dots<c_{j-1}<c_j<\dots$ and by induction $c_j\le b_j$ for $j\ge0$. Consequently, 
\be
\label{prf_lem_DLVP_1}
\sum_{j\ge0}c_j \int_{\{2^{-j-1}<f \le 2^{-j}\}} f\,d\mu<\oo.
\ee
Notice also that by induction there holds, 
\be\label{prf_lem_DLVP_2}
c_{j+i}\le2^{i/2}c_j\qquad\text{for }i,j\ge 0.
\ee
Moreover, 
\be\label{prf_lem_DLVP_3}
c_j=\prod_{i=1}^{j} \min\lt(\sqrt 2\,,\dfrac{b_i}{b_{i-1}}\rt)\quad{\st{j\up\oo}\longto}\ \oo.
\ee
Indeed, denoting $\Lambda:=\lt\{i\ge1:b_i\ge \sqrt 2 b_{i-1}\rt\}$.
If on the one hand $\Lambda$ is finite, then for $j\ge j_0:=\max \Lambda$, $c_j=(c_{j_0}/b_{j_0})b_j$ hence $c_j\to\oo$ as $j\to\oo$. If on the other hand, $\Lambda$ is infinite, there holds
\[
c_j\ge\lt(\sqrt 2\rt)^{|\Lambda\cap[1,j]|}\ \st{j\up\oo}\longto\ \oo.
\]
Summing up, we have $1=c_0<\dots< c_j<c_{j+1}<\dots \to\oo$ and~\eqref{prf_lem_DLVP_1},\eqref{prf_lem_DLVP_2}\&\eqref{prf_lem_DLVP_3} hold.
Let us define $g:(0,\oo)\to[1,\oo)$ by
\[
g(s):=
\begin{cases}
\phantom{cc}c_j&\text{for }s\in(2^{-j-1},2^{-j}],\ j\ge1,\\
c_0=1&\text{for }s>1/2.
\end{cases}
\]
We notice that $g$ is nonincreasing and that by~\eqref{prf_lem_DLVP_3} $g(s)\to\oo$ as $s\dw0$. Let us set, for $s\ge0$,
\[
h(s):=\int_0^s g(t)\,dt,
\]
(observe that by~\eqref{prf_lem_DLVP_2}, $g(s)\le s^{-1/2}$ for $0<s\le1/2$ so that $h$ is well-defined).
We have that $h:\R_+\to\R_+$ is continuous, concave, increasing and such that $h(0)=0$ and $h'(0^+)=\oo$. Eventually, we compute for $j\ge1$ and $s\in (2^{-j-1},2^{-j}]$, 
\begin{multline*}
h(s)\,\le h(2^{-j}) = \sum_{i\ge j}2^{-i-1}c_i =2^{-j-1}\sum_{i\ge 0}2^{-i}c_{j+i}  \\
\st{\eqref{prf_lem_DLVP_2}}\le 2^{-j-1}c_j\sum_{i\ge 0}2^{-i/2}=(2+\sqrt2) 2^{-j-1}c_j\le(2+\sqrt2) s c_j.
\end{multline*}
Consequently, $h\circ f \le (2+\sqrt2)c_j f$ in the domain $\{2^{-j-1}<f \le 2^{-j}\}$. We conclude by~\eqref{prf_lem_DLVP_1} that $h\circ f$ is $\mu-$integrable which proves the lemma.
\end{proof}

\section{A proof of Theorem~\ref{thm_decomp} assuming that $G=(\R,+,|\cd|)$}
\label{App2}
We give here an alternative proof of Theorem~\ref{thm_decomp}, in the spirit to that of~\cite[Theorem~1]{ACMM}, assuming that $G$ is boundedly compact.  In fact for concreteness we assume that $G=\R$ and to avoid technicalities that $A$ is compactly supported. The important point is that the additional assumption ensures that the closure/compactness property~\eqref{notassumed} holds true.  As it is essential to Step~1 below this shows that a new approach were needed to deal with the general case.

\begin{proof}
Let $A\in\NN_k^\R(\R^n)$ be rectifiable and compactly supported. We introduce the set
\[
\DD:=\lt\{A_j \text{ set-decomposition of }A \text{ such that } \N(A_j) \text{ is nonincreasing}\rt\}.
\]
This set is not empty as it contains $(A,0,\dots)$. Let us endow the space of sequences $v_j\in\R$ indexed by $j\ge1$ with the lexicographic ordering, \textit{i.e.} $(v'_j)\sslo(v_j)$ if there exists $j_0\ge1$ such that $v'_j=v_j$ for $1\le j<j_0$ and $v'_{j_0}<v_{j_0}$. We consider the optimization problem
\be\label{prf_thm_decR_1}
(v_j):=\inf_{\text{\scriptsize lex.}} \lt\{(\N(A_j)) : A_j\in\DD\rt\}.
\ee
Since the $\N(A_j)$'s are nonnegative, the infimum is well defined and $v_j\ge0$ for every $j$.

 We claim that if $A_j$ is a minimizer  of~\eqref{prf_thm_decR_1} then each $A_j$ is  set-indecomposable. To see this, we assume by contradiction that for some $j_0\ge1$, $A_{j_0}$ admits a nontrivial set-decomposition $(A_{j_0}',A_{j_0}'')$, that is $\max(\N(A_{j_0}'),\N(A_{j_0}''))<\N(A_{j_0})$. Substituting $(A_{j_0}',A_{j_0}'')$ for $A_{j_0}$ in the sequence $A_j$ and then rearranging the terms in decreasing order of $\N$-norms we obtain a set-decomposition $\tilde A_j$ of $A$ with $\tilde A_j= A_j$ for $j<j_0$ and $\N(\tilde A_{j_0})<\N(A_{j_0})$. This contradicts the minimality of $A_j$.\medskip

To complete the proof we establish that~\eqref{prf_thm_decR_1} does admit a minimizer. \medskip

\noindent
\textit{Step 1. ($\F$-compactness of minimizing sequences).}

Let $(A_j^m)_{m\ge1}$ be a minimizing sequence for~\eqref{prf_thm_decR_1}. Let us first fix $j\ge1$. The sequence $(A_j^m)_m$ satisfies $\N(A_j^m)\le\N(A)$ and $\supp A_j^m\sub\supp A$ so by~\eqref{notassumed} there exists a normal chain $A_j$ such that up to extraction, $A_j^m\to A_j$ in $\F$-norm. Using a diagonal argument, we may assume that $A_j^m\to A_j$ as $m\up\oo$ in $\F$-norm for every $j\ge1$. Moreover, by lower semicontinuity of the masses under $\F$-convergence,
\be\label{prf_thm_decR_2} 
\N(A_j)\le\liminf_{m\up\oo}\N({A_j^m}).
\ee
By Lemma~\ref{lem_DLVP} there exists a cost function $h\in C(\R_+,\R_+)$, increasing, concave such that $h(0)=0$, $h'(0^+)=\oo$ and $\M_h(A)<\infty$ (in particular, $h$ is strictly subadditive). We deduce from~\cite[Propositions~2.6\&2.7]{CdRMS2017} that the $A_j$'s are rectifiable and
\be\label{prf_thm_decR_3} 
\M_h(A_j)\le\liminf_{m\up\oo}\M_h({A_j^m}).
\ee
Moreover by Lemma~\ref{lem_isopineq} there exists a nondecreasing function $\eta:\R_+\to\R_+$ with $\eta(m)\dw0$ as $m\dw0$ such that~\eqref{eta_intro} holds with $A'=A^m_j$ for any $j,m\ge1$.\medskip

\noindent
\textit{Step 2. (Uniform $\F$-summability and mass identities).}
Let $j_0,m\ge1$. Since $\N(A)=\sum_j\N(A_j^m)$ and $(\N(A_j^m))_j$ is nonincreasing, we have for $j\ge j_0$,
\[
\M(A_j^m)\le\N(A_j^m)\le\N(A_{j_0}^m)\le\dfrac1{j_0}\sum_{i=1}^{j_0}\N(A_i^m)\le\dfrac{\N(A)}{j_0}.
\]
Using~\eqref{eta_intro}, we compute,
\begin{multline*}
\sum_{j\ge j_0}\F(A_j^m)
\le\sum_{j\ge j_0} \eta\lt(\M(A_j^m)\rt) \lt(\N(A_j^m)+\M_h(A_j^m)\rt)\\
\le\eta\lt(\dfrac{\N(A)}{j_0}\rt)\sum_{j\ge j_0}\lt(\N(A_j^m)+\M_h(A_j^m)\rt) 
 \le\eta\lt(\dfrac{\N(A)}{j_0}\rt)\lt(\N(A)+\M_h(A)\rt) \quad\st{j_0\up\oo}\longto\ 0.
\end{multline*}
Hence, the series $\sum_j A_j^m$ converges in $\F$-norm uniformly with respect to $m$. As a consequence, we can pass to the limit in $A=\sum_jA^m_j$ and deduce the identity,
\be\label{prf_thm_decR_4} 
A=\sum_{j\ge1} A_j.
\ee
Then, by the triangle inequality for $\N$ and $\M_h$ (see~\eqref{count_sub}), there holds,
\be\label{prf_thm_decR_5.5} 
\N(A)\le\sum_{j\ge1}\N(A_j)\qquad\qquad\text{and}\qquad\qquad\M_h(A)\le\sum_{j\ge1}\M_h(A_j).
\ee
By definition, we have for $m\ge1$, $\N(A)=\sum_j\N(A_j^m)$ and $\M_h(A)=\sum_j\M_h(A_j^m)$. Together with~\eqref{prf_thm_decR_2},~\eqref{prf_thm_decR_3}, Fatou  and~\eqref{prf_thm_decR_5.5}, this leads to
\be\label{prf_thm_decR_5} 
\N(A)=\sum_{j\ge1}\N(A_j)\qquad\qquad\text{and}\qquad\qquad\M_h(A)=\sum_{j\ge1}\M_h(A_j).
\ee
Thus~\eqref{prf_thm_decR_2}  improves to $\N(A_j)=\lim_{m\up\oo}\N({A_j^m})$ and we get eventually, 
\be\label{prf_thm_decR_6} 
\N(A_j)=v_j\quad\text{for every }j\ge1, \text{ where the }v_j\text{'s are given by }\eqref{prf_thm_decR_1}.
\medskip 
\ee
\noindent
\textit{Step 3. (Conclusion by strict subadditivity of $h$).} 

At this point, we know that the $A_j$'s are normal rectifiable chains satisfying~\eqref{prf_thm_decR_4},~\eqref{prf_thm_decR_5} and~\eqref{prf_thm_decR_6}. To conclude that the sequence $A_j$ is a minimizer of~\eqref{prf_thm_decR_1} we still have to show that it is a set-decomposition of $A$.

Let $\Sigma$ be a countably $k$-rectifiable set of $\R^n$ such that for every $j\ge1$, $A_j=A_j\restr\Sigma$ and let us write
\[
A=w\xi\h^k\restr\Sigma\qquad\text{and}\qquad A_j=w_j\xi\h^k\restr\Sigma\quad\text{for }j\ge1,
\] 
where $w$, $w_j$ are Borel measurable functions on $\R^n$ and $\xi$ is a Borel measurable field of unit $k$-vectors orienting $\Sigma$. From~\eqref{prf_thm_decR_4}\&\eqref{prf_thm_decR_5} we have $w(x)=\sum w_j(x)$ for $\h^k$-almost every $x\in \Sigma$. Using the fact that $h$ is increasing and subadditive, we compute, 
\begin{multline*}
\M_h(A)=\int_\Sigma h\lt(|w|\rt)\, d\h^k =\int_\Sigma h\lt(\lt|\sum_{j\ge1}w_j\rt|\rt)\, d\h^k\\
\le  \int_\Sigma h\lt(\sum_{j\ge1}\lt|w_j\rt|\rt)\,d\h^k\le \sum_j  \int_\Sigma h\lt(\lt|w_j\rt|\rt)\,d\h^k=\sum\M_h(A_j).
\end{multline*}
By~\eqref{prf_thm_decR_5} the inequalities are identities and since $h$ is strictly subadditive, we conclude that for $\h^k$-almost every $x\in\Sigma$, there exists $j_0\ge 1$ such that $|w_{j_0}(x)|=|w(x)|$ and $w_j(x)=0$ for $j\ne j_0$. Since $\sum w_j=w$ $\h^k$-almost everywhere on $\Sigma$, we have in fact $w_{j_0}(x)=w(x)$. Hence $A_j$ is a set-decomposition of $A$ which proves the result.   
\end{proof}

\subsection*{Acknowledgments}
M. Goldman is partially supported by the ANR SHAPO. B. Merlet is partially supported by the INRIA team RAPSODI and the Labex CEMPI (ANR-11-LABX-0007-01).

\bibliographystyle{alpha}
\bibliography{BibStripes}

\begin{thebibliography}{CDRMS17}

\bibitem[ACMM01]{ACMM}
L.~Ambrosio, V.~Caselles, S.~Masnou, and J.-M. Morel.
\newblock Connected components of sets of finite perimeter and applications to
  image processing.
\newblock {\em J. Eur. Math. Soc. (JEMS)}, 3(1):39--92, 2001.

\bibitem[AFP00]{Am_Fu_Pal}
L.~Ambrosio, N.~Fusco, and D.~Pallara.
\newblock {\em Functions of bounded variation and free discontinuity problems}.
\newblock Oxford Mathematical Monographs. The Clarendon Press, Oxford
  University Press, New York, 2000.

\bibitem[AK00]{AK2000}
L.~Ambrosio and B.~Kirchheim.
\newblock Currents in metric spaces.
\newblock {\em Acta Math.}, 185(1):1--80, 2000.

\bibitem[Alm86]{Almgren}
F.~Almgren.
\newblock Deformations and multiple-valued functions.
\newblock {\em Geometric measure theory and the calculus of variations (Arcata,
  Calif., 1984)}, 44:29--130, 1986.

\bibitem[AM17]{AlMa2017}
G.~Alberti and A.~Massaccesi.
\newblock On some geometric properties of currents and {F}robenius theorem.
\newblock {\em Atti Accad. Naz. Lincei Rend. Lincei Mat. Appl.},
  28(4):861--869, 2017.

\bibitem[AM22]{flat}
G.~Alberti and A.~Marchese.
\newblock On the structure of flat chains with finite mass.
\newblock {\em preprint}, 2022.

\bibitem[Amb90]{Ambrosio1990}
L.~Ambrosio.
\newblock Metric space valued functions of bounded variation.
\newblock {\em Ann. Scuola Norm. Sup. Pisa Cl. Sci. (4)}, 17(3):439--478, 1990.

\bibitem[BDNP22]{BdNP2022}
P.~Bonicatto, G.~Del~Nin, and E.~Pasqualetto.
\newblock Decomposition of integral metric currents.
\newblock {\em J. Funct. Anal.}, 282(7):Paper No. 109378, 28, 2022.

\bibitem[BPR20]{BPR2020}
P.~Bonicatto, E.~Pasqualetto, and T.~Rajala.
\newblock Indecomposable sets of finite perimeter in doubling metric measure
  spaces.
\newblock {\em Calc. Var. Partial Differential Equations}, 59(2):Paper No. 63,
  39, 2020.

\bibitem[BW18]{BraWir}
A.~Brancolini and B.~Wirth.
\newblock General transport problems with branched minimizers as functionals of
  1-currents with prescribed boundary.
\newblock {\em Calc. Var. Partial Differential Equations}, 57(3):1--39, 2018.

\bibitem[CDRM21]{colombo2021well}
M.~Colombo, A.~De~Rosa, and A.~Marchese.
\newblock On the well-posedness of branched transportation.
\newblock {\em Communications on Pure and Applied Mathematics}, 74(4):833--864,
  2021.

\bibitem[CDRMS17]{CdRMS2017}
M.~Colombo, A.~De~Rosa, A.~Marchese, and S.~Stuvard.
\newblock On the lower semicontinuous envelope of functionals defined on
  polyhedral chains.
\newblock {\em Nonlinear Anal.}, 163:201--215, 2017.

\bibitem[CFM19]{CFM2019a}
A.~Chambolle, L.A.D. Ferrari, and B.~Merlet.
\newblock Variational approximation of size-mass energies for {$k$}-dimensional
  currents.
\newblock {\em ESAIM Control Optim. Calc. Var.}, 25:Paper No. 43, 39, 2019.

\bibitem[Cho23]{Chou22}
H-C Chou.
\newblock Integral decompositions of varifolds.
\newblock {\em Ann. Global Anal. Geom.}, 64(1):3, 2023.

\bibitem[DPH12]{HdP2012}
T.~De~Pauw and R.~Hardt.
\newblock Rectifiable and flat {$G$} chains in a metric space.
\newblock {\em Amer. J. Math.}, 134(1):1--69, 2012.

\bibitem[DPH14]{HdP2014}
T.~De~Pauw and R.~Hardt.
\newblock Some basic theorems on flat {$G$} chains.
\newblock {\em J. Math. Anal. Appl.}, 418(2):1047--1061, 2014.

\bibitem[Fed69]{Federer}
H.~Federer.
\newblock {\em Geometric measure theory}.
\newblock Die Grundlehren der mathematischen Wissenschaften, Band 153.
  Springer-Verlag New York Inc., New York, 1969.

\bibitem[FF60]{FedFlem60}
H.~Federer and W.~H. Fleming.
\newblock Normal and integral currents.
\newblock {\em Ann. of Math. (2)}, 72:458--520, 1960.

\bibitem[{F}le66]{Fleming66}
W.~H. {F}leming.
\newblock {Flat chains over a finite coefficient group.}
\newblock {\em {Trans. Am. Math. Soc.}}, 121:160--186, 1966.

\bibitem[GM22]{GM_tfc}
M.~Goldman and B.~Merlet.
\newblock Tensor rectifiable {$G$}-flat chainsd.
\newblock {\em ArXiv}, 2022.

\bibitem[Men16]{Menne2016}
Ulrich Menne.
\newblock Weakly differentiable functions on varifolds.
\newblock {\em Indiana Univ. Math. J.}, 65(3):977--1088, 2016.

\bibitem[Whi99a]{White1999-1}
B.~White.
\newblock The deformation theorem for flat chains.
\newblock {\em Acta Math.}, 183(2):255--271, 1999.

\bibitem[Whi99b]{White1999-2}
B.~White.
\newblock Rectifiability of flat chains.
\newblock {\em Ann. of Math. (2)}, 150(1):165--184, 1999.

\bibitem[You18]{Young2018}
R.~Young.
\newblock Quantitative nonorientability of embedded cycles.
\newblock {\em Duke Math. J.}, 167(1):41--108, 2018.

\end{thebibliography}

\end{document}